\documentclass[11pt,tbtags,reqno]{article}
\usepackage[margin=1.3in]{geometry}
\usepackage{amssymb}
\usepackage{amsmath}
\usepackage{amsthm}
\usepackage[swedish, english]{babel}
\usepackage{psfrag}   
\usepackage{graphicx}
\usepackage[retainorgcmds]{IEEEtrantools}
\usepackage{bbm}
\usepackage[square, comma, numbers, sort&compress]{natbib}
\usepackage{enumerate}
\usepackage{hyperref}
\usepackage[colorinlistoftodos, textsize=tiny, backgroundcolor=yellow
, linecolor=green]{todonotes}
\usepackage{sidecap}
\usepackage{float}

\newcommand{\E}{\mathbb E}
\newcommand{\R}{\mathbb R}

\newcommand{\C}{\mathcal C}
\newcommand{\D}{\mathcal D}
\newcommand{\F}{\mathcal F}

\newcommand{\ep}{\varepsilon}
\newcommand{\eps}{\varepsilon}

\newcommand{\ud}{\,\mathrm{d}}
\newcommand{\vd}{\mathrm{d}}
\newcommand{\lt}{\left}
\newcommand{\rt}{\right}

\newcommand{\Dr}{X} % a letter to denote the unknown drift

\newcommand{\supp}{\mathop{\mathrm{supp}}\nolimits}
\newcommand{\Var}{\mathop{\mathrm{Var}}\nolimits}

\numberwithin{equation}{section}
\theoremstyle{plain}

\newtheorem{theorem}{Theorem}[section]
\newtheorem{lemma}[theorem]{Lemma}
\newtheorem{corollary}[theorem]{Corollary}
\newtheorem{proposition}[theorem]{Proposition}

\newtheorem{assump}[theorem]{Assumption}
\newtheorem{remark}[theorem]{Remark}
\theoremstyle{definition}
\newtheorem*{acknowledgement}{Acknowledgement}

\begin{document}

\title{Bayesian sequential least-squares estimation for the drift of a Wiener process}
%\author{Erik Ekstr\"om, Ioannis Karatzas and Juozas Vaicenavicius}

\author{Erik Ekstr\"om\thanks{Department of Mathematics, Uppsala University, 75106 Uppsala, Sweden \mbox{(email: \href{mailto: ekstrom@math.uu.se}{ekstrom@math.uu.se})}},
Ioannis Karatzas\thanks{
Department of Mathematics, Columbia University, 2990 Broadway, New York, NY 10027, USA \mbox{(email: 
\href{mailto: ikl@columbia.edu}{ikl@columbia.edu})}; 
and \textsc{Intech} Investment Management, One Palmer Square, Suite 441, Princeton, NJ 08542, USA
\mbox{(email: \href{mailto: ikaratzas@intechjanus.com}{ikaratzas@intechjanus.com})}. Support from the National Science Foundation under grant NSF-DMS-14-05210 is gratefully acknowledged.
} and Juozas Vaicenavicius\thanks{Department of Information Technology, Uppsala University, 75105 Uppsala, Sweden \mbox{(email: \href{mailto: juozas.vaicenavicius@it.uu.se}{juozas.vaicenavicius@it.uu.se})}}}

\maketitle
\begin{abstract}

\smallskip
\noindent
Given a Wiener process with unknown and  unobservable drift, we try to estimate this drift as effectively but also as quickly as possible, in the presence of a quadratic penalty for the estimation error and of a fixed, positive cost per unit of observation time.  In a Bayesian framework, where the unobservable drift is assumed to have a known ``prior" distribution, this question reduces to choosing judiciously a stopping time  for an appropriate  diffusion process in natural scale. We  establish structural properties  of the solution for the corresponding  problem of optimal stopping. In particular,  we show that, regardless of the prior distribution, the continuation region is monotonically shrinking in time; and provide conditions on the prior distribution   guaranteeing a one-sided stopping region. Finally, we illustrate the theoretical results through a   detailed study of some concrete prior distributions.

\end{abstract}



\section{Introduction}

Imagine   trying to estimate a quantity about which there is considerable uncertainty, and which cannot be observed directly. Instead, one has access to a stream of observations that this unobservable quantity affects  and, based on this stream,   tries to find an  estimator of the unobservable quantity which is ``optimal" in the sense of least-squares. However, access to the stream of information is costly: one pays a fixed, positive cost per unit of time, for as long as information is being obtained. How does one  then resolve the dilemma inherent in this situation, which calls for balancing the conflicting requirements of fidelity in estimation  and of cost minimization?

We study here an instance of this problem in a highly idealized and stylized form, and in a Bayesian setting. Namely, we assume that the unobservable quantity is a   random variable $X$ with known distribution, and that one observes sequentially 
the   process 
\begin{IEEEeqnarray}{rCl} 
\label{E:Y}
Y(t)=X t + W(t)\,, \qquad 0 \le t < \infty\,.\end{IEEEeqnarray} 
Here $W$ is a standard Wiener process, independent of the random variable $X$. Moreover, %adopting a Bayesian  methodology, 
we assume that  the  known,  ``prior" distribution $\mu$ of $X$ has finite second moment  and is  {\it non-atomic;} that is,  we exclude the trivial case where $\mu$ is a one-point distribution. We posit that, at any given time $t \in [0, \infty)$, we have  access to the observations %information sigma-algebra 
\[\sigma\{Y(s),0\leq s\leq t\}.\]
The right-continuous augmentation $\, \mathbb F^Y=\{\mathcal F^Y(t)\}_{0\leq t<\infty}\,$  of the family  of $\sigma-$algebras  $\, \big( \sigma\{Y(s),0\leq s\leq t\} \big)_{0\leq t<\infty}\,$ is called  the {\em observations filtration}, and we set 
$$\,
 {\cal F}^Y (\infty) \, := \, \sigma \bigg( \bigcup_{\,0 \le t < \infty} {\cal F}^Y (t) \bigg)\,.
$$
We denote by $\mathcal T^Y$   the collection  of   stopping times of this filtration $\, \mathbb F^Y$; to wit, the collection of random variables $\tau : \Omega \to [0, \infty)$ with %the property 
$\{ \tau \le t \} \in {\cal F}^Y (t)$   for every $t \in [0, \infty)$.  

\smallskip
Based on the flow of information $\mathbb F^{\,Y} $, we construct  the least-squares estimate 
\begin{equation}\label{11a}
\widehat X(t)\,=\,\E \big[ X\,\vert \,\mathcal F^Y(t)\big]\,, \qquad 0\leq t<\infty
\end{equation}
of the unobserved variable. In this work we seek to compute the minimal expected cost
\begin{IEEEeqnarray}{rCl} 
\label{E:OSP}
C_* = \inf_{\tau \in \mathcal T^Y} C(\tau)\,, \qquad C(\tau) := \E \big[\big(X-\widehat X(\tau)\big)^2+c\tau \big],
\end{IEEEeqnarray} 
and to determine  whether it is attained by some stopping time $\,\tau_* \in \mathcal T^Y$. 
Here  $c>0$ is a given real constant, representing the cost of one unit of delay in the estimation procedure. The positivity of this constant, along with the obvious bound $C_* \le C(0) = \text{Var} (X) < \infty$, implies that we may restrict attention in (\ref{E:OSP}) to stopping times $\tau$ with 
\begin{equation}
 \label{13}
\E [\tau] < \infty.
\end{equation}

\noindent{\bf Preview:}
We show in Section~\ref{sec2} that the least-squares estimate $\widehat X$ in \eqref{11a}
is an It\^o diffusion in natural scale, and describe its dynamics in detail. The question of \eqref{E:OSP} is cast in Section~\ref{sec3} as a problem of optimal stopping for this process $\widehat X$, with a cost criterion that involves only the constant $c>0$ and the local variance function of $\widehat X$. This piece of serendipity allows us to obtain quite general qualitative properties of the solution, as developed in Section~\ref{sec5}. 
In particular, we show that the continuation region contracts in time for any given prior distribution, and we provide conditions under which the continuation region and the stopping region are separated by a single curve.
Some explicit results for the Gaussian and the Bernoulli cases are presented in Sections~\ref{sec3} and \ref{Bernoulli}, respectively. Finally, Section~\ref{sec6} presents an additional example in which the prior distribution is symmetric.

\begin{acknowledgement}
We are heavily indebted to Lane Chun Yeung for the interest he took in this work, the many suggestions, comments and corrections he made, and the additional examples he provided. 

Many thanks go also to Miguel-\'Angel Garrido, Daniel Lacker, Martin Larsson, Johannes Ruf and Walter Schachermayer for the many critical comments to this work that helped us sharpen our thinking; and to the two referees, for spotting errors and helping us clarify and sharpen the exposition of the paper.
\end{acknowledgement}

%%%%%%%%%
\section{Preliminaries on the conditional mean and variance processes}
%%%%%%%%%
\label{sec2}

In this section  we recall a general result  regarding the conditional mean and variance processes from the theory  of filtering. We build then on this result  in order to unveil the structure of our problem at hand,  as well as the stochastic dynamics of the processes that are crucial for its analysis. 

%%%%%%%
\subsection{Projecting onto the observations filtration}
%%%%%%%

We first recall the conditional distribution of $X$ given observations on the process $Y$. 
For a proof we refer to \cite[Proposition 3.16]{BC09}. 

\begin{proposition}
\label{T:integrals}
Consider a function $q:\R\to\R$ satisfying the integrability condition $$\int_\R  \, \big|q(u)\big| \,\mu(\vd u) < \infty\,.$$ Then, for any $t\geq0$, we have 
\begin{IEEEeqnarray*}{rCl} 
\E\lt[ q(\Dr) \vert \mathcal F^Y(t)\rt] 
=\frac{\int_\R q(u) \, \exp\{uy -u^2t/2\} \, \mu(\vd u)}{\int_\R \, \exp\{uy -u^2t/2\} \, \mu(\vd u)} \,\bigg|_{y=Y(t)}\,.
\end{IEEEeqnarray*} 
\end{proposition}

\bigskip
On the strength of  Proposition~\ref{T:integrals}, the conditional expectation of $X$ given the observations  up to   time  $t \in (0, \infty)$ is given as  
\begin{IEEEeqnarray}{rCl} 
\label{E:hX}
\widehat X(t) := \E[ X \, | \, \F^Y(t)]=G(t,Y(t))\,,
\end{IEEEeqnarray} 
  where for each $(t,y) \in [0, \infty) \times \R$ we set 
\begin{IEEEeqnarray}{rCl} 
\label{E:G(t,y)}
G(t,y) :=\frac{\int_\R u \, \exp\{uy -u^2t/2\} \, \mu(\vd u)}{\int_\R \, \exp\{uy -u^2t/2\} \, \mu(\vd u)}
=\int_\R u \,\mu_{t,y}(\vd u)
\end{IEEEeqnarray}
and
\begin{IEEEeqnarray}{rCl}
\label{E:mu}
\mu_{t,y}(A)\,:=\,\frac{\,\int_A \,  \exp\{uy -u^2t/2\} \,  \mu(\vd u)}{\int_\R \, \exp\{uy -u^2t/2\} \, \mu(\vd u)}\,, \qquad A \in {\cal B} (\R)\,.
\end{IEEEeqnarray}
This measure $\mu_{t,y}$ is the conditional (``posterior") distribution of $X$ at time $t$,  given the values $Y(s), \, 0 \le s < t$ and $Y (t)=y$ of the observation process up to that time.  

We have a similar computation for the conditional variance 
\[
\Var \big(X\vert \,\mathcal F^Y(t)\big)=\E \big[ \big(X-\widehat X(t) \big)^2\,\big \vert \, \mathcal F^Y(t) \big]=\E \big[X^2\, \big \vert \, \mathcal F^Y(t)\big]-\widehat X^2(t)=H(t,Y(t))
\]
of $X$ given the observations  up to   time  $t \in (0, \infty)$, where
\begin{IEEEeqnarray}{rCl}
\label{E:H(t,y)}
H(t,y)&\,:=\,& \frac{\int_\R u^2 \, \exp\{uy -u^2t/2\} \, \mu(\vd u)}{\int_\R \, \exp\{uy -u^2t/2\} \, \mu(\vd u)}-
\left(\frac{\int_\R u \, \exp\{uy -u^2t/2\} \, \mu(\vd u)}{\int_\R \, \exp\{uy -u^2t/2\} \, \mu(\vd u)}\right)^2  \\
&=& \int_\R u^2 \mu_{t,y}(\vd u)-\left( \int_\R u\, \mu_{t,y}(\vd u) \right)^2= \int_\R \left( u- G\big( t,y  \big)   \right)^2 \mu_{t,y}(\vd u) . \nonumber
\end{IEEEeqnarray}
It is straightforward from \eqref{E:G(t,y)}, \eqref{E:H(t,y)} that the quantities $G(t, y)$, $H(t, y)$  are, respectively,  the center of gravity and the second central moment of the measure $\mu_{t,y}$ in \eqref{E:mu}; to wit,  
$$
 G(t, y) = \E \big[  X \, \big \vert \,Y(s), \, 0 \le s < t;\, Y(t)=y  \big], \quad H(t, y) = \text{Var} \big[  X \, \big \vert \,Y(s), \, 0 \le s < t;\,Y(t)=y  \big].
$$

The function $G $ of \eqref{E:G(t,y)} is strictly increasing in its spatial variable, has partial derivatives of all orders on $(0,\infty)\times\R$, and   satisfies on this strip the {\it Backwards Burgers equation}
\begin{equation}
\label{GPDE}
%\left(
\partial  G+\frac{1}{2}\, D^2G + G\cdot D G  %\right)(t,y)
=0\,.
\end{equation}
The  gradient of the function $G$, i.e., the function $$H\, =\, D G$$ of \eqref{E:H(t,y)}, is    positive    on the strip $(0,\infty)\times\R$  and satisfies there    the equation
\begin{equation}
\label{HPDE}
%\left(
\partial  H+\frac{1}{2}\, D^2 H + G\cdot D H + H^2  %\right)(t,y)
=0\,.
\end{equation}
Here and throughout, we are denoting by $ \, \partial \equiv \partial / \partial t \,$ the partial derivative with respect to the  temporal argument $t$, and by $\, D^k \equiv \partial^k / \partial y^k \,$ the partial derivative of order $\,k=1,2, \cdots\,$ with respect to the spatial argument, in this case $y$.

%\medskip
%\noindent
%{\it Remark 2.2: Bijections.}
\begin{remark} 
\label{bijections}
{%\em 
(Bijections.)}
{\rm  Let now $\mathcal{I}_\mu$ denote the interior of the smallest closed interval containing the support   of the probability measure $\mu$, i.e.,
\begin{IEEEeqnarray}{rCl} 
\label{E:supp}
\mathcal{I}_\mu \,=\, \big(\inf(\mathcal{S}_\mu), \,\sup(\mathcal{S}_\mu) \big)\qquad \text{with} \qquad \mathcal{S}_\mu := \supp (\mu)\,.
\end{IEEEeqnarray}
Then, for any given $t \in [ 0, \infty)$, the function $$G_t(\cdot) \, \equiv \, G(t,\cdot): \R \to \mathcal{I}_\mu\,$$ defined in \eqref{E:G(t,y)} is a 
strictly increasing, continuous bijection (see also \cite{EV}).  The strict increase of this function $G_t(\cdot) %\equiv G(t, \cdot)
$ implies that   $Y(t) = G_t^{-1} \big(\widehat{X}(t)\big)$ holds for $0 \le t < \infty$. 

To wit, the observation processes $Y$ and the least-squares estimate process $\widehat{X}$ are bijections of each other pointwise in time, and thus generate the same filtration. In particular,    $    G_t^{-1}(x)$  is the unique value of the observation process $Y(t)$ at time $t$, that yields $\widehat X(t)=x$. 
}
\end{remark}

%\medskip
%\noindent
%{\it Remark 2.3: The Widder Transform.} 
\begin{remark}\label{widder}
{%\em 
(The Widder Transform.)}
{\rm 
The derivation of the parabolic backwards partial differential equations \eqref{GPDE}, \eqref{HPDE} is facilitated by the observation that $G$ is itself the logarithmic gradient $ G = D \log F$ of the function 
\begin{IEEEeqnarray}{rCl} 
\label{E:F(t,y)}
F(t,y) \,:= \int_\R \, \exp \Big\{uy - \frac{t}{2}\, u^2  \Big\} \, \mu(\vd u) \,, \qquad (t,y) \in (0,\infty)\times\R\,
\end{IEEEeqnarray}
  that appears in the denominators of \eqref{E:G(t,y)}, \eqref{E:H(t,y)}. It is checked easily that this function, the so-called {\it Widder Transform} of the prior distribution $\mu$,  solves the backward heat equation $$\partial  F+\frac{1}{2}\, D^2 F\,=\,0\,.$$ Conversely, as shown by Widder 
\cite{W44} and 
 Robbins  \& Siegmund  \cite{RS73} , every positive solution of this backward heat  equation can be written in the form \eqref{E:F(t,y)}, in terms of an appropriate measure $\mu$ on %the Borel sets of the real line
${\cal B} (\R)$.  For a probabilistic treatment and development of the  relevant theory, see Section 3.4.B in \cite{KS1}.
}
\end{remark}

%\bigskip
For technical convenience, we shall impose henceforth the following integrability condition:

\begin{assump}
\label{ass}
For some real number $a>0$, the prior distribution $\mu$ satisfies 
\begin{equation}
\label{integrability}
\int_\R \, \exp \{a u^2 \}\,    \mu(\vd u)<\infty\,.
\end{equation}
\end{assump}

\medskip
Assumption~\ref{ass} is a rather mild requirement   in the sense that, for any $t>0$, the integrability condition \eqref{integrability} is satisfied by the posterior distribution $\mu_{t,y}$   of any prior $\mu$.  %\footnote{~ Yes. Still, I would like us to include a section, or a remark, which works out the details of how this assumption can be removed. (IK). EE: I agree.} 
The assumption allows us to extend the definition of $\mu_{t,y}$ in \eqref{E:mu} above, to include
points of the type $(0,y)$; and the resulting $\mu_{0,y}$ coincides with the posterior distribution in a scenario with 
prior distribution 
\[
\xi(A)\,:= \, \frac{\, \int_A \,\exp \{a u^2 /2 \}\,   \mu(\vd u)\,}{\,\int_\R \, \exp \{a u^2 /2 \}\,  \mu(\vd u)\,}\,,\qquad A \in {\cal B} (\R)
\] 
conditional on  observing $Y(a)=y$. Consequently, the points $(0,y)$ can be regarded 
as interior points for a shifted problem started instead at time $-a$;  it is therefore clear that, 
for instance, \eqref{GPDE}  holds then on the whole domain $[0, \infty)\times\R$.
%\footnote{~ Should we also assume that the variance of $\mu_{0,y}$ is bounded (in $y$)? (EE)}

%%%%%%%%%%%%%%%%
\subsection{Dynamics under the observations filtration}
%%%%%%%%%%%%%%%%%

The process
\begin{IEEEeqnarray}{rCl} \label{E:Innov}
\widehat W(t) := \,Y(t)  - \int_0^t \widehat X(s)   \ud s  \, = \int_0^t \big(X - \widehat X(s) \big)  \ud s + W(t)\, , \qquad 0 \le t < \infty\,,
\end{IEEEeqnarray}
known as the {\it innovation process} in the theory of filtering, is clearly adapted to the observations filtration $\mathbb{F}^{Y}$. It is also a  Wiener process  of this filtration, as it is continuous, an $\,\mathbb{F}^{Y}-$martingale, and has the right quadratic variation; for instance,  see \cite[Proposition 2.30 on p.~33]{BC09}. 

We write $\,\mathbb F^{\,\widehat W}=\big( \F^{\, \widehat W}(t) \big)_{0\leq t<\infty}\,$ for the right-continuous augmentation of the filtration 
$\sigma \big\{\widehat W(s) : 0\leq s \leq t \big\} \,, \,0\leq t<\infty \,$ that this process generates; similarly, we shall use the notation  $\mathbb F^{\,\widehat X}=\big( \F^{\, \widehat X}(t) \big)_{0\leq t<\infty}$ for the right-continuous augmentation of the filtration generated by  
the conditional expectation  process $\widehat X$ in \eqref{E:hX}, and $\mathbb F^{\,Y}=\big( \F^{\, Y}(t) \big)_{0\leq t<\infty}$ for the right-continuous augmentation of the filtration generated by the observation process $Y$. Clearly, % from the above, 
and in light of Remark~\ref{bijections}, we have the comparisons $\, \mathbb F^{\,\widehat W} \subseteq \mathbb F^{\,Y} = \mathbb F^{\,\widehat X} \,.$
%\footnote{~ Should $\mathbb F^Y$ be the {\em right-continuous augmentation} of the filtration generated by $Y$?}

\smallskip
We deduce now from \eqref{E:hX}, \eqref{E:Innov} the representation for the observations process 
\begin{IEEEeqnarray}{rCl} \label{E:SDEY}
\ud Y(t)=\widehat X(t)\ud t + \ud \widehat W(t) = G \big( t, Y(t) \big) \ud t + \ud \widehat W(t)
\end{IEEEeqnarray}
as the solution of a stochastic differential equation driven by the innovations process $\widehat W$, with initial condition $Y(0)=0$. Because of the smoothness of the function $G$, this equation admits a pathwise unique, strong solution, so we deduce the filtration identities
\begin{IEEEeqnarray}{rCl} 
\label{E:Filtr_Ident}
\mathbb F^{\,\widehat W} =\mathbb F^{\,Y} =\mathbb F^{\,\widehat X}\,.
\end{IEEEeqnarray}
On the other hand, with the notation $G_t(\cdot)  = G(t, \cdot)$ already introduced in Remark~\ref{bijections}, we set
\begin{IEEEeqnarray}{rCl} 
\label{E:psi}
\Psi(t, x) \, := \, D G \big(t, G_t^{-1}(x)\big) \,  = \, H   \big(t, G_t^{-1}(x)\big).
\end{IEEEeqnarray}
An application of It\^o's formula to \eqref{E:hX}    yields, in conjunction with \eqref{GPDE} and \eqref{E:SDEY}, a stochastic differential equation for the conditional mean process $\widehat X $ of \eqref{E:hX}, namely, 
\begin{IEEEeqnarray}{rCl} 
\label{E:SDEX}
\ud \widehat X(t) \,=\, \Psi \big( t, \widehat X(t) \big) \ud \widehat{W}(t)\,, \qquad \widehat X(0) =\, \E (X)=\int_\R u\, \mu(\vd u)  \, .
\end{IEEEeqnarray}

The function $\Psi$ of \eqref{E:psi}, the dispersion function of the stochastic differential equation   right above, can be expressed as 
\[ \Psi(t,x) =  H\big(t,G_t^{-1}(x) \big)= \text{Var} \big[  X \, \big \vert \,\widehat{X}(s), \, 0 \le s < t\,;\,\widehat{X}(t)=x  \big]\,, \]
and a bit more generally 
\begin{IEEEeqnarray}{rCl} 
\label{E:VARt}
 \Psi \big(t,\widehat{X} (t) \big)\,=\, 
  H\big(t,G_t^{-1}(\widehat{X} (t)) \big)\,=\,H \big( t, Y(t) \big)\,=\,  \text{Var} \big[  X \, \big \vert \,  {\cal F}^Y (t)  \big].
\end{IEEEeqnarray}

Furthermore, it is checked with the help of \eqref{GPDE}, \eqref{HPDE} that the function $\Psi >0$ of \eqref{E:psi} satisfies on the strip $(0,\infty)\times \mathcal{I}_\mu$  the fully nonlinear, backwards parabolic equation   
\begin{equation}
\label{PsiPDE}
 \partial  \Psi + \Psi^2 \left(  \frac{1}{2} \, % \frac{\partial^2  \Psi}{\partial x^2} 
 D^2 \Psi +1  \right) 
 =0\,.
\end{equation}
Once again, we denote differentiation with respect to the temporal argument $t$   by $\partial$, and differentiation with respect to the spatial argument (in this case  $x$)  by $D$. 

Finally, we recall from \cite[Proposition 3.6]{EV} the following result about the function $\Psi$.

\begin{proposition}[Properties of the dispersion function $\Psi$] \item \label{T:psigr}
\begin{enumerate} 
\item\label{T:psigr1}
$\partial  \Psi \le 0\,;$  consequently, by \eqref{PsiPDE}, we have $\,D^2 \Psi \geq -2\,$.
\item
If $\mu$ is compactly supported, then the function $\Psi$ is bounded.
\end{enumerate} 
\end{proposition}

\section{Optimal Stopping}
\label{sec3}

The above considerations show that the optimal stopping problem \eqref{E:OSP} can be cast in the form 
\begin{IEEEeqnarray}{rCl}
 \label{E:OSP2}
\inf_{\tau\in \mathcal T%^{\widehat X}
} \, \E \, \Big[ \Psi \big(\tau,\widehat X(\tau)\big) +c\tau \, \Big]\,.
\end{IEEEeqnarray} 
Here  $\Psi$ is the function of \eqref{E:psi},    the process $\widehat X$ satisfies the dynamics of \eqref{E:SDEX}, and $ \mathcal T%^{\widehat X} 
$ stands for the collection of stopping times of the filtration $\mathbb F^{\,\widehat X}=\mathbb F^{\,\widehat W} =\mathbb F^{\,Y}$, as in \eqref{E:Filtr_Ident}. 

\medskip
It is a noteworthy feature of this problem, that the same function $\Psi$ of (\ref{E:psi}) appears both as the dispersion of the diffusion $\widehat{X}$ in (\ref{E:SDEX}), {\it and} as the cost function for the new formulation of the optimal stopping problem in (\ref{E:OSP2}). This feature makes the problem rather special, and aids considerably  its analysis in Sections 4, 5. 

\begin{proposition}
\label{reformulation}
For every stopping time $\tau\in\mathcal T%^{\widehat X}
$ we have 
\begin{IEEEeqnarray}{rCl} 
\label{E:equivalence}
\E \big[\Psi(\tau,\widehat X(\tau)\big)\big] = \Var (X) -\E\left[\int_0^\tau \Psi^2 \big(s,\widehat X(s) \big)\ud s\right].
\end{IEEEeqnarray}
\end{proposition}

\begin{proof}
From \eqref{E:Y} and the strong law of large numbers for the Wiener process, we have $ \lim_{t \to \infty} \big( Y(t) / t \big) = X$, a.e.; in other words,  the random variable $X$ is ${\cal F}^Y (\infty)$-measurable. As a result, the P.\,L\'evy martingale convergence theorem gives $\lim_{t \to \infty} \E \big( X^k | {\cal F}^Y (t) \big)  = \E \big( X^k | {\cal F}^Y (\infty) \big) = X^k$ a.e.,    for $k=1,2$; therefore  also 
$$
\lim_{t \to \infty}  \Psi \big(t,\widehat X (t)\big)\,=\, \lim_{t \to \infty} \bigg(  \E \big( X^2 \, | \,{\cal F}^Y (t) \big)  - \Big( \E \big( X \,  | \,{\cal F}^Y (t) \big) \Big)^2 \bigg) \,=\,0 
$$
  on the strength of  \eqref{E:VARt}.   Now it follows from the dynamics in \eqref{E:SDEX},   the partial differential equation in \eqref{PsiPDE}, and  elementary stochastic calculus, that the positive process
\begin{eqnarray*}
M(t) &:=& \,\Psi \big(t,\widehat X (t)\big) + \int^t_0 \Psi^2 \big(s,\widehat X (s)\big)\ud s\\
&=&  \text{Var} (X) + \int^t_0 \Psi  \big(s,\widehat X (s)\big)D \Psi  \big(s,\widehat X (s)\big)\ud \widehat{W}(s), \quad 0 \le t < \infty
\end{eqnarray*}
is a local martingale. It is thus also a supermartingale, and consequently 
\begin{equation}
\label{32a}
\E \left( \Psi \big(\tau,\widehat X (\tau)\big) + \int^\tau_0 \Psi^2 \big(s,\widehat X (s)\big)\ud s \right) \le \text{Var} (X) <\infty
\end{equation}
holds by the optional sampling theorem  for every stopping time $ \tau \in {\cal T} $; this  includes $\tau = \infty$, so we have also 
\begin{equation}
\label{32b}
\E \left( \int^\infty_0 \Psi^2  (s,\widehat X (s) )\ud s \right) \le \text{Var} (X)<\infty\,.
\end{equation}
We shall show presently that, as claimed in \eqref{E:equivalence}, the first inequalities 
in \eqref{32a} and \eqref{32b} hold actually as equalities. 

In order to see these things, let us start from the observation that the representation 
\begin{IEEEeqnarray}{rCl} 
\label{E:SDEXext}
 \widehat X (\tau) \,=\, \E (X) +  \int^\tau_0  \Psi  \big(s,\widehat X (s)\big)\ud \widehat{W}(s) 
\end{IEEEeqnarray} 
from \eqref{E:SDEX} holds for every stopping time $\tau \in {\cal T}$, including $\,\tau = \infty\,$: the martingale $\widehat X$ and the submartingale $\widehat X^2$ are both uniformly integrable. Thus, the representation 
$$X-\widehat X (\tau) \,=\,   \int_\tau^\infty   \Psi  \big(s,\widehat X (s)\big)\ud \widehat{W}(s)$$ holds, as does the analogue 
\begin{eqnarray*}
\text{Var} \big( X \,\big| \, {\cal F}^Y (\tau) \big) &=& \E \Big[ \big( X - \widehat{X}(\tau ) \big)^2 \, \big| \, {\cal F}^Y (\tau) \Big] = \E \left( \int_\tau^\infty \Psi^2 \big(s,\widehat X (s)\big)\ud s\, \Big| \, {\cal F}^Y (\tau) \right)\\
&=& \Psi \big( \tau,   \widehat{X}(\tau ) \big)
\end{eqnarray*}
of \eqref{E:VARt}; for the second equality we have used the finite upper bound of \eqref{32b}. This, in turn, leads to 
$$
\E \Big[ \text{Var} \big( X \,\big| \, {\cal F}^Y (\tau) \big) \Big]=\, \E   \big( X - \widehat{X}(\tau ) \big)^2    = \,\E \, \Psi \big( \tau,   \widehat{X}(\tau ) \big)=\, \E   \int_\tau^\infty   \Psi^2 \big(s,\widehat X (s)\big)\ud s
$$
upon taking expectations. In addition, \eqref{E:SDEXext} gives 
$$
\text{Var} \Big(  \E \big( X \,\big| \, {\cal F}^Y (\tau) \big) \Big) \,=\, \text{Var}  \big( \widehat{X}(\tau )  \big)=\, \E    \int^\tau_0      \Psi^2 \big(s,\widehat X (s)\big)\ud s\,,
$$
once again using the finite upper bound in \eqref{32b}.
From these computations, and from a classical identity about variances, we deduce
$$
\text{Var}  \big( X  \big)\,=\,\E \Big[ \text{Var} \big( X \,\big| \, {\cal F}^Y (\tau) \big) \Big] + \text{Var} \Big( \E \big( X \,\big| \, {\cal F}^Y (\tau) \big) \Big)
$$
$$
~~~~~~~\,=\, \E \left( \Psi \big( \tau,   \widehat{X}(\tau ) \big) + \int^\tau_0 \Psi^2 \big(s,\widehat X (s)\big)\ud s \right),
$$
that is, our claim \eqref{E:equivalence}. With the choice $\tau = \infty$, these considerations give the identity  $ \E  \int^\infty_0 \Psi^2  (s,\widehat X (s) )\ud s  = \text{Var} (X)$, as claimed.
\end{proof}

The identity \eqref{E:equivalence} allows us now to cast the optimal stopping problem  of \eqref{E:OSP}/\eqref{E:OSP2} in the equivalent form
\begin{IEEEeqnarray}{rCl} \label{E:equivalentOSP}
v:=\inf_{\tau\in \mathcal T %\atop    \E (\tau) < \infty %^{\hat W}
}  \E\left[\int_0^\tau \Big( c-\Psi^2(s,\widehat X(s)) \Big) \ud s\right].
\end{IEEEeqnarray}

%%%%%%%%
\subsection{Markovian framework}
\label{markov}
%%%%%%%%

To study the optimal stopping problem in its new form \eqref{E:equivalentOSP}, we first embed it   % the problem 
into a Markovian framework,
by allowing the diffusion $\widehat X$ of \eqref{E:SDEX} to start at any given point $(t,x)\in[0,\infty)\times \mathcal{I}_\mu$. More precisely, we define the function $v:[0,\infty)\times \mathcal{I}_\mu\to (-\infty, 0]$ via
\begin{IEEEeqnarray}{rCl} 
\label{E:v}
v(t,x)\,:=\,\inf_{\tau\in\mathcal T%^{\widehat W}
} \, \E\left[\int_0^\tau \Big( c-\Psi^2 \big( t+s,\widehat X^{t,x}(t+s) \big) \Big)\ud s\right],
\end{IEEEeqnarray}
where the dynamics of the process $\widehat X=\widehat X^{(t,x)}$ are given by
\begin{IEEEeqnarray*}{rCl}
\left\{\begin{array}{l}
\ud \widehat X(t+s) = \Psi \big( t+s, \widehat X(t+s)\big) \ud \widehat{W}(s)\\
\widehat X(t)=x\end{array}\right.
\end{IEEEeqnarray*}
and $\widehat W$ is again standard scalar Brownian motion.
Since $\tau=0$ is an admissible stopping time, the value function $v$ in \eqref{E:v} is non-positive:   $v\leq 0$. On the other hand, it is clear from (3.1), (3.2) that $\, v(t,x) \ge -\text{Var}(X) > - \infty\,$, so $v$ is also real-valued, as indicated. 

In accordance with standard optimal stopping theory for Markov processes with continuous paths, we introduce the so-called continuation region
\[\mathcal C:=\{ (t,x)\in [0,\infty)\times \mathcal{I}_\mu  : \, v(t,x)<0\}\]
and its complement, the stopping region 
\[\mathcal D:=\{(t,x)\in [0,\infty)\times \mathcal{I}_\mu :\, v(t,x)=0\}.\]
Moreover, for any given starting point $(t,x)$, we denote by 
\[
\tau^{(t,x)}:=\inf \big\{s\geq 0: \big(t+s, \widehat X^{(t,x)}(t+s) \big)\in\mathcal D \big\}
\]
the first hitting time of the stopping region.
Then we know (for instance,  \cite{Fak}, \cite{vM76}, \cite{Shir}, \cite{ElK}) that the function $v:[0,\infty)\times \mathcal{I}_\mu\to[-\text{Var}(X),0]$ of \eqref{E:v} is upper-semicontinuous, and that for each $(t,x)\in [0,\infty)\times \mathcal{I}_\mu $ the stopping time $\tau^{(t,x)}$ attains  the infimum there, i.e., %\footnote{~ I hope I have said nothing but the truth, here. Please check. (IK). 
%EE: I am not sure how to present this. Of course I believe in continuity, but it is not directly stated in these references (as far as I know). 
%Perhaps upper semi-continuity is sufficient? On the other hand, maybe I worry too much: continuity in $x$ is more or less the flow property, (uniform) continuity in $t$ should be ok since we have an integral. Need to get back to this...} 
\[
v(t,x)=\E\left[\int_0^{\tau^{(t,x)}} \Big( c-\Psi^2 \big(t+s,\widehat X^{(t,x)}(t+s)\big) \Big) \ud s\right].
\]

\begin{remark}
\label{stop_now}
{\rm
It is clear from the formulation \eqref{E:v} that immediate stopping ($\tau^{(t,x)}=0$) is optimal, if the inequality 
\begin{IEEEeqnarray}{rCl}
\label{stop_reg}
c\geq \sup_{(t,x)\in[0,\infty)\times \mathcal I_\mu}\Psi^2(t,x)=\sup_{x  \in \mathcal I_\mu}\Psi^2(0,x) 
\end{IEEEeqnarray}
holds; here, the equality follows from Proposition~\ref{T:psigr}. A bit more generally, if 
\[
c\geq \sup_{x \in\mathcal I_\mu}\Psi^2(T_c,x)
\]
holds for some $T_c\in(0,\infty)$, then the strip $[T_c,\infty)\times \mathcal I_\mu$ belongs to the stopping region %\subseteq 
$\mathcal D$.
}
\end{remark}

%%%%%%%%
\subsection{A very simple special case: The Gaussian distribution}
\label{gauss}
%%%%%%%%

As the simplest illustration, let us consider the Gaussian prior distribution $\mu$ with mean $m\in\R$ and variance $\sigma^2\in(0,\infty)$, i.e.,
\[
\mu(\vd u)=\frac{1}{\sqrt{2\pi \sigma^2}}\, \exp\bigg\{-\frac{(u-m)^2}{2\sigma^2} \bigg\}\, \vd u\,,
\]
  a special case of the Kalman-Bucy filter. 
Here we have $\mathcal{I}_\mu=\R$, and the functions $\, F$, $G$, $H$ and $\Psi$ take the respective forms
$$
F(t,y)\,=\, \frac{1}{%\sqrt{2\pi\sigma^2}
\sqrt{1+ \sigma^2 t}}\, \exp\bigg\{ \,- \frac{1}{2\sigma^2} \, \bigg( \frac{(m + \sigma^2 y)^2}{1+\sigma^2 t}- m^2\bigg) \bigg\}
$$
\[
G(t,y) \,= \,\frac{m+\sigma^2y}{1+\sigma^2 t}\,, \qquad 
H(t,y)\,=\, \Psi(t,x)\,=\,\frac{\sigma^2}{1+\sigma^2 t}\,=\,:\xi(t)\,.
\]
Now, the function $\,t \mapsto c-\xi^2(t)\,$ is negative for $ t \in \big[0,\frac{1}{\sqrt{c\,}\,}-\frac{1}{\sigma^2}\big)$ if $\,\sqrt{c\,}<\sigma^2$, and it is everywhere non-negative if
$\,\sqrt{c\,}\geq \sigma^2$. With $$\tau_*\,:= \,\bigg(\frac{1}{\sqrt{c\,}\,}-\frac{1}{\sigma^2}\bigg)^+ ,$$ it follows from Remark \ref{stop_now} that the above constant
$\tau_*$ is an optimal (albeit trivial!) stopping time in \eqref{E:equivalentOSP}.

In \cite{CNS}, a similar result is obtained in the case when $W$ is a fractional Brownian motion.

%%%%%%%%%%%%%%%%%%%%%%
\section{A time-homogeneous case: the Bernoulli Distribution
}
%%%%%%%%%%%%%%%%%%%%%%%%%%
\label{Bernoulli}

As our second example, let us   consider now the Bernoulli prior distribution 
\[\mu=(1-p)\delta_{-\beta}+p\delta_{\beta}\]
with symmetric support,
where $p\in(0,1)$ and $\beta\in(0,\infty)$. 
%Since the problem is translation invariant (in the prior), and since we embed the problem so as to consider all possible starting points simultaneously, the solution below also covers the case of any (not necessarily symmetric) Bernoulli distribution.
In this case we have $\mathcal{I}_\mu = (- \beta, \beta)$, as well as 
$$
G(t, y) \,=\, \beta \, \frac{\, p \, e^{\beta y}- (1-p) \, e^{-\beta y} \, }{\, p \, e^{\beta y} + (1-p) \, e^{-\beta y} \, }\,, \qquad \text{thus} \qquad H(t,y) \,=\, \beta^2 - G^2 (t,y)
$$
and 
\[\Psi(t,x)\,=\,\beta^2-x^2 \,=\,:\psi(x)\,.\]
We note that we are here at the opposite extreme of the example in subsection~\ref{gauss}: All these are functions %$\Psi(\cdot,\cdot)$ 
of only   the spatial variable; and the last of them does not even depend on the parameter $p \in (0,1)$. 

The stopping problem \eqref{E:equivalentOSP} thus takes the form
\begin{IEEEeqnarray}{rCl}
\label{bernoulli_value}
v(x)\,=\,
\inf_{\tau\in \mathcal T%^{\hat W}
}  \, \E\left[\int_0^\tau \Big( c-\psi^2 \big(\widehat X(t)\big) \Big) \ud t  \right]
\end{IEEEeqnarray}
where $\widehat{X}$ is a diffusion in  natural scale, with state-space $\,\mathcal{I}_\mu = (- \beta, \beta)$ and  initial condition  $\,x \in \mathcal{I}_\mu\,$: % inaccessible boundaries: 
\begin{IEEEeqnarray}{rCl}
\label{bernoulli}
\left\{\begin{array}{l}
\ud \widehat X(t) = \psi \big( \widehat X(t) \big) \ud \widehat{W}(t)\\
~~\widehat X(0)=\beta(2p-1)=:x \in \mathcal{I}_\mu\,.\end{array}\right.
\end{IEEEeqnarray}
We note that for $\beta^4\leq c$, the integrand in \eqref{bernoulli_value} is non-negative, and hence the trivial stopping time $\tau_* \equiv 0$
is optimal. 

Thus, we assume from now onwards that $$\beta^4\,> \, c  \,\,;$$ then $c-\psi^2(x)$ is negative for $| x| <  \gamma\,$ with  $$\, \gamma \,:= \,\sqrt{\beta^2-\sqrt {c\,}\,}\,,$$ zero for $|x|= \gamma$, and positive for $\, \gamma < |x| < \beta\,$. %In particular, and on account of \eqref{stop_reg},  the intervals $\, [\gamma, \beta)\,,$ $\, ( - \beta, -\gamma]$ with  are both included in the optimal stopping region. 
Conjecturing that
an optimal stopping rule is of the type 
\begin{equation}
\label{E:opt_stop_time}
\tau_a^* \,:= \, \inf \big\{t\geq 0:\vert \widehat{X} (t) \vert\geq a \big\}
\end{equation}
for some constant $a\in(\gamma,\beta)$, general optimal stopping theory leads to the following free-boundary problem:  

\smallskip
{\it To find a  constant $a\in(\gamma,\beta)$ and an evenly symmetric  function $u:(-\beta,\beta)\to(-\infty,0]$ of class $\C^1 \big((-\beta,\beta)\big)\cap\C^2 \big((-\beta,\beta)\setminus\{-a,a\}\big),$ such that
\begin{equation}
\label{u}
\left\{\begin{array}{lll}
u(x)<0\,, & ~\big(\psi^2(x) /2\big)u''(x)+c-\psi^2(x)=0\,; & ~~x\in [0,a),\\
u(x)=0\,, & ~\,c-\psi^2(x) > 0\,; &~~ x\in [a,\beta);\end{array}\right.
\end{equation}
and then to argue that the function $u$ coincides with the minimum expected cost $v$ in \eqref{bernoulli_value}.} 	

\medskip
In the two paragraphs that follow  we shall show that this problem admits a unique solution, which coincides with the value function $v$ of (\ref{bernoulli_value}) and can be computed explicitly.

%%%%%%%%%%%%%%
\subsection{Verification}
%%%%%%%%%%%%%%

Indeed, if such a function $u$ with the above properties exists, the process 
$$
N \,:= \,u \big( \widehat{X}\big) - u(x) - \int_0^\cdot \frac{1}{2} \big( \psi^2u^{\prime \prime}  \big) \big( \widehat{X} (t)\big) \mathrm{d} t \,= \int_0^\cdot  \big(  \psi u^{\prime  }  \big) \big( \widehat{X} (t)\big) \mathrm{d}\widehat{W} (t)
$$
is a local martingale. The function $\psi u^{\prime  }$ is continuous, and supported on the compact interval $[-a,a]$, thus bounded. Therefore, for any stopping time $\tau \in {\cal T}$ with $\,\E (\tau) < \infty$ as in \eqref{13}, we have 
$$
\E \big( N^2 (\tau) \big) \,=\,\E \int_0^\tau \big( \psi u^{\prime  }  \big)^2 \big( \widehat{X} (t)\big) \mathrm{d}t \,\le\, \big| \big|  \psi u^{\prime  }  \big| \big|_\infty^2 \,\,\E (\tau)\, <\, \infty \,.
$$
As a consequence, $N (\cdot \wedge \tau)$ is a square-integrable martingale, and    $\E \big( N  (\tau) \big) =0$ holds, leading to 
\begin{equation}
\label{E:comp}
u(x) \,=\, \E \big[ u \big( \widehat{X} (\tau)\big) \big]- \E\int_0^\tau \frac{1}{2} \big(  \psi^2u^{\prime \prime} \big) \big( \widehat{X} (t)\big) \mathrm{d} t \,\le \,\E   \int_0^\tau \Big( c-\psi^2 \big(\widehat X(t)\big) \Big) \ud t
\end{equation}
on account of the inequalities $\, u \le 0\,,$ $\, \big( \psi^2 / 2 \big) u^{\prime \prime} + c - \psi^2 \ge 0\,$ from \eqref{u}. 

\smallskip
We repeat now the above reasoning for the stopping time $\,\tau^*_a\,$ defined in \eqref{E:opt_stop_time}. This satisfies the property $\,\E (\tau^*_a) < \infty\,,$ as is checked by considering the diffusion process $\widehat{X}$ of \eqref{bernoulli} on the interval $[-a, a]$ as its state-space, and recalling Proposition 5.5.32\,(i) in \cite{KS1}. For this stopping time, both inequalities summoned to justify the last comparison in \eqref{E:comp} hold as equalities, and thus so does \eqref{E:comp} itself: 
\begin{equation}
\label{E:comp_too}
u(x)   \,= \,\E   \int_0^{\tau^*_a} \Big( c-\psi^2 \big(\widehat X(t)\big) \Big) \ud t\,.
\end{equation}
Now \eqref{E:comp} and \eqref{E:comp_too} show that the stopping time $\,\tau^*_a\,$ is optimal for the problem of \eqref{bernoulli_value}, among all stopping times with finite expectation. As we argued in the discussion following \eqref{E:OSP}, these are the only relevant  times for the stopping problem under consideration, and we are done: $u (x) = v(x)$ holds for every $x \in {\cal I}_\mu$. 

In particular, there can exist at most one solution to the free-boundary problem.

%%%%%%%%%%%%%%
\subsection{Construction}
%%%%%%%%%%%%%%

For a given constant $a \in (0, \infty)$, the recipe 
\begin{equation}
\label{E:Construct}
u(x) \,:=\, 2 \int_x^a \bigg( \int_y^a \frac{\,\psi^2 (\xi)-c\,}{\psi^2 (\xi)}\, \ud \xi \bigg) \ud y\,, \qquad 0 \le x \le a
\end{equation}
defines a function that satisfies the equation $\,\big(\psi^2(x) /2\big)u''(x)+c-\psi^2(x)=0\,$ in \eqref{u}, as well as the ``smooth-fit" conditions" $\,u(a) =  u^\prime (a-) =  0.$ 

We extend this function by even symmetry to all of $[-a,a]$. For the resulting extension to have the claimed smoothness, we need the condition $\,   u^\prime (0+) =  0\,,$ namely 
\begin{equation}
\label{E:FPP}
\int_0^a \frac{\,\ud \xi %-\psi^2 (\xi) 
\,}{\,\psi^2 (\xi)\,}\,  \,=\,\frac{\,a\,}{c}\,.
\end{equation}
Now, the function $ Q(x) := \int_0^x \psi^{-2} (y) \big( c- \psi^{2} (y) \big) \ud y\,, ~ 0<x<\beta\,$ satisfies $\, Q(0)=0$, decreases strictly on $(0, \gamma)$, and increases strictly to infinity on $(  \gamma, \beta)$. It attains its overall minimum at $x=\gamma$, namely, 
$$
Q(\gamma) = \int_0^\gamma \frac{\,c-\psi^2 (\xi) \,}{\psi^2 (\xi)}\, \ud \xi \,<\,0\,.
$$
 Therefore, there exists a unique number $a \in (\gamma, \beta)$  that satisfies $Q( a) =0$, i.e., \eqref{E:FPP}.
 
 With the constant $a$ thus chosen, $\,c-\psi^2(x) > 0\,$ holds for every $\, x\in [a,\beta)$; setting
 \begin{equation}
\label{E:Construct_too}
u(x) \,:=\, 0 \,, \qquad x \in (a,\beta)
\end{equation}
and extending again by even symmetry, we obtain a function $u$ defined via \eqref{E:Construct}, \eqref{E:Construct_too} on all of $\,{\cal I}_\mu = ( - \beta, \beta)$; this function    satisfies all the requirements of the free-boundary problem in \eqref{u}. From what we have proved so far, the function $u$ emerges as the unique solution of this problem, as well as the minimum expected cost  in \eqref{bernoulli_value}; that is, $v \equiv u$.  

%%%%%%%%%%%%%%%
\section{Structural properties}
%%%%%%%%%%%%%%%
\label{sec5}

In contrast to the two examples just discussed, the typical situation is that the stopping and continuation regions cannot be described so easily. Thus, general methods to determine their structural properties are of considerable interest. 

For this purpose, the following monotonicity result will prove useful. It is based on the observation,  that the term $\Psi^2$ appearing in the integrand in \eqref{E:v} coincides with the instantaneous quadratic 
variation rate of the underlying process $\widehat X$ (we extend the function $\Psi$  to be equal to zero outside 
$[0,\infty)\times \mathcal I_{\mu})$. This suggests a time-change of the martingale $\widehat X$  
in the manner of Dambis-Dubins-Schwarz (e.g., Theorem 3.4.6 and Problem 3.4.7 in \cite{KS1}). We follow the  construction given  by Janson and Tysk in \cite{JT}, where two diffusion processes in natural scale with the same starting point, but with different dynamics, are constructed as time-changes of the {\em same} Brownian motion.

\begin{theorem}\label{monotonicity}
Assume that two distributions $\mu_i$, $i=1,2$ are given and that the corresponding variance functions $\Psi_i$ in \eqref{E:psi} satisfy
$\,\Psi_1(t,x)\geq \Psi_2(t,x)\,$ for all $(t,x)\in[0,\infty)\times\R$. Then the corresponding value functions $v_i$, $i=1,2$ of \eqref{E:v} satisfy
$\,v_1(t,x)\leq v_2(t,x)\,$ for all $(t,x)\in[0,\infty)\times\R$.
\end{theorem}

\begin{proof} It suffices to show that $v_2 (0,x)\geq v_1 (0,x)$; the case of a general time variable is similar.
For $x\in \R$, let $B$ be a one-dimensional Brownian motion with $B(0)=x$. For each $i=1,2$, let $\tau_i (\cdot)$, $\,i=1,2$, be the unique
stopping time solution of  the integral equation 
\[
\tau_i(t)=\int_{0}^t\Psi_i^2 \big( \theta ,B(\tau_i(\theta))\big)\ud \theta, \quad\quad   0 \le t < \infty \,,\]
with the terminology and construction developed by Janson and Tysk in \cite{JT}.
Then the process $X_i(t):=B(\tau_i(t))$, $\,0 \le t < \infty\,$ satisfies the stochastic integral equation 
\[
  X_i(s)=x+\int_0^s\Psi_i \big(u, X_i(u) \big)\,\vd B_i(u)\,, \qquad 0 \le s < \infty 
\] 
for some Brownian motion $B_i$. Consequently, the distribution  of $\{X_i(s),\,\,  s\geq 0\}$ coincides with the distribution    of $\{\widehat X_i^{(0,x)}(s), \,\, s\geq 0\}$, for $i=1,2$. 

Furthermore, it follows from  \cite[Lemma 10]{JT} that
\begin{equation}
\label{ineq}
\tau_1(t)\geq \tau_2(t)
\end{equation}
holds for all $t\geq 0$. Now, let $\gamma_2$ be a stopping time (of the right-continuous augmentation of the filtration generated by the 
process $X_2$) which minimizes 
\[
\E\left[ c \gamma -\int_{0}^{\gamma}\Psi_2^2 \big(s,X_2(s) \big)\ud s\right]=\E \big[ c \gamma -\tau_2(\gamma)\big]
\]
over all stopping times $\gamma$. Define 
\[
\gamma_1\,:= \,\inf \big\{s\geq 0:\tau_1(s)>\tau_2(\gamma_2) \big\}
\]
so that $\tau_1(\gamma_1)=\tau_2(\gamma_2)$, and note that \eqref{ineq} implies that 
$\gamma_1\leq \gamma_2$. Then $\gamma_1$ is a stopping time for the process $X_1$,
though not necessarily an optimal one for the stopping problem under consideration. Consequently, 
\begin{IEEEeqnarray*}{rCl} 
v_2(0,x) &=& \E\left[ c \gamma_2 -\int_{0}^{\gamma_2}\Psi^2_2\big(s, X_2(s) \big)\ud s\right]
= \E\left[ c \gamma_2 -\tau_2(\gamma_2)\right]\\
&\geq& \E\left[ c \gamma_1 -\tau_1(\gamma_1)\right]
= \E\left[ c \gamma_1 -\int_{0}^{\gamma_1}\Psi_1^2 \big(s,X_1(s)\big)\ud s\right]
\geq v_1(0,x),
\end{IEEEeqnarray*}
which completes the proof.
\end{proof}

As we have seen in the Gaussian and Bernoulli cases, the structure of the stopping region $\mathcal D$ depends crucially on the prior distribution $\mu$; however, we note the following consequence of Theorem~\ref{monotonicity}, which provides a very general structural result with respect to the temporal parameter. 

\begin{corollary}\label{time-increase}{\bf (Contracting continuation region.)}
The function $t \mapsto v( t,x)$ is non-decreasing, for every fixed $x\in \mathcal{I}_\mu$. Consequently, the $t$-section of the stopping region, namely,  
\[
\mathcal D_t \,:= \, \big\{x\in  \mathcal{I}_\mu \,:\,(t,x)\in\mathcal D\big\}\,,
\]
is increasing in time: $\mathcal D_{t_1}\subseteq \mathcal D_{t_2}\,$ for 
$\,0\leq t_1\leq t_2$.
\end{corollary}

\begin{proof}
Consider two time points $t_i$, $i=1,2$ with $t_1<t_2$, and define $\Psi_i(t,x)=\Psi(t_i+t,x)$ for $(t,x)\in[0,\infty)\times\mathcal I_\mu$.
Since,   on the strength of Proposition~\ref{T:psigr}, the function $\Psi(\cdot,x)$ is decreasing,  we have $\Psi_1(t,x)\geq \Psi_2(t,x)$  for each 
$(t,x)\in [0,\infty)\times\mathcal{I}_\mu$. It then follows from Theorem~\ref{monotonicity} 
that the corresponding value functions
satisfy $v_1(0,x)\leq v_2(0,x)$, which is equivalent to $v(t_1,x)\leq v(t_2,x)$.

Thus $x \in {\cal D}_{t_1}$ (i.e., $v(t_1, x)=0$) leads to $v(t_2, x)=0$, i.e., to $x \in {\cal D}_{t_2}$.
\end{proof}

\begin{corollary}{\bf (Comparison with the Bernoulli distribution.)}
For $\beta>0$, let $a=a(\beta)>0$ be the optimal stopping boundary-point for the Bernoulli distribution with support $\{-\beta,\beta\}$ 
as determined in Section~\ref{Bernoulli}. For a ``prior" distribution $\mu$, recall the notation of (\ref{E:supp}). 
\begin{itemize}
\item[(i)]
Assume that $\mathcal{S}_\mu\subseteq [-\beta,\beta]$. Then $\mathcal C\subseteq [0,\infty)\times(-a,a)$.
\item[(ii)]
Assume that $\mathcal{S}_\mu\subseteq (-\infty,-\beta]\cup[\beta,\infty)$, with $\mathcal{S}_\mu\cap(-\infty,-\beta]\not=\emptyset$
and $\mathcal{S}_\mu\cap[\beta,\infty)\not=\emptyset$. Then $\mathcal C\supseteq [0,\infty)\times(-a,a)$.
\end{itemize}
\end{corollary}

\begin{proof}
It is straightforward to check that among all distributions with support contained in $[-\beta,\beta]$ and expected value $x\in[-\beta,\beta]$, the Bernoulli distribution
\begin{equation}
\label{bern_sp}
\frac{\beta-x}{2\beta} \, \delta_{-\beta}+\frac{ \beta + x}{2\beta} \, \delta_{\beta}
\end{equation}
   is the one with the largest variance.
Consequently, if $\,\mathcal{S}_\mu\subseteq [-\beta,\beta]$, then $\,\Psi(0,x)\leq \beta^2-x^2=\psi(x)$. Thus, by Proposition~\ref{T:psigr} we have $\Psi(t,x)\leq \Psi(0,x)\leq \psi(x)$,
and (i) follows from Theorem~\ref{monotonicity}.

Similarly, among all distributions $\mu$ with $\mathcal S_\mu\cap(-\beta,\beta)=\emptyset$ and with expected value $x\in(-\beta,\beta)$, the one with the smallest variance is the Bernoulli distribution in \eqref{bern_sp}. Consequently, $\Psi(t,x)\geq  \beta^2-x^2$ for all $x\in\mathcal I_\mu$, and 
(ii) follows as above, on account of Theorem~\ref{monotonicity}. 
\end{proof}

We  restrict now attention to sub-classes of prior distributions, for which further structural properties can be derived.
We first recall the following well-known result from optimal stopping theory (see for instance \cite[Remark, page 217]{bO07}).

\begin{lemma}\label{T:locallygood}
Assume that $\Psi^2(t,x)>c$ at some point $(t,x)\in[0,\infty)\times \mathcal{I}_\mu$. Then $(t,x)\in\mathcal C$.
\end{lemma}

\begin{proof}
By the continuity of the function $\Psi^2$, there exists a real number $\eps>0$ and a rectangle $\mathcal R=[t_1,t_2)\times(a,b)\subseteq [0,\infty)\times \mathcal{I}_\mu$ with
$(t,x)\in\mathcal R\,,$ and $\,\Psi^2-c>\eps\,$ on $\mathcal R$. Denoting by 
\[\tau_{\mathcal R}:=\inf \big\{s\geq 0: \big(t+s,\widehat X(t+s)\big)\notin\mathcal R \big\},
\]
we have 
\[v(t,x)\leq \E\left[ \int_{0}^{\tau_\mathcal R}\left(c- \Psi^2(t+s, \widehat X(t+s) \right)\ud s\right]\leq -\eps \, \E [\tau_{\mathcal R}]<0,\]
which shows that $(t,x)\in\C$.
\end{proof}

 Our next task is to provide conditions, under which the stopping region is one-sided. We shall use the notation  $\overline{ \mathcal I}_{\mu}=\mathcal I_{\mu}\cup\{a,b\}\,,$ where $a=\inf(\mathcal S_\mu)$ and $b=\sup(\mathcal S_\mu)$ are the (possibly infinite) boundary points of $\,\mathcal I_\mu \,$, as in   \eqref{E:supp}.% and $\mathcal S_\mu  = \supp (\mu)$.

\begin{proposition}\label{one-sided}{\bf (One-sided stopping region.)}
Assume that, for every fixed time $t\geq 0$, the function $x\mapsto \Psi(t,x)$ $($equivalently, the function $y\mapsto H(t,y))$  is non-decreasing.
Then the following statements hold:
\begin{itemize}
\item[(i)]
There exists a non-decreasing function $b:[0,\infty)\to \overline{\mathcal I}_\mu$ such that the optimal continuation region is of the form $$\,\mathcal C\,=\, \big\{(t,x)\in [0,\infty)\times \mathcal{I}_\mu:x>b(t) \big\}\,.$$
\item[(ii)]
With $\Psi(t,\infty):=\lim_{x\to\infty}\Psi(t,x)$,   let $$T\,:=\,\inf \big\{t\geq 0:\Psi^2(t,\infty)\leq c \big\}\,.$$ Then $\,b(t) \in \mathcal I_\mu \,$ for all $t<T,$ and $\,b(t)=\sup(\mathcal S_\mu )$ for $t\geq T$, as in \eqref{E:supp}. 
\item[(iii)]
If $x\mapsto \Psi(t,x)$ is strictly increasing for all $t\geq 0$, then the function $b:[0,\infty)\to \overline{\mathcal I}_\mu$ is continuous.  
\end{itemize}
\end{proposition}

\begin{proof}
(i). Without loss of generality, we consider the initial time $t=0$. We consider two points $(0,x_1)$ and $(0,x_2)$ with $x_1,x_2\in\mathcal I_\mu$ and $x_1<x_2$. 
By comparison results for solutions of stochastic integral equations (see for instance  \cite[Theorem IX.3.7]{RY}), we obtain 
$\widehat X^{0,x_1}(s)\leq \widehat X^{0,x_2}(s)$ for all times $s\geq 0$.
Therefore, 
\begin{IEEEeqnarray*}{rCl}
\E\left[\int_0^{\tau} \Big( c-\Psi^2(s,\widehat X^{0,x_1}(s)) \Big) \ud s\right]  \,  \geq  \,  \E\left[\int_0^{\tau} \Big( c-\Psi^2(s,\widehat X^{0,x_2}(s)) \Big) \ud s\right]
\end{IEEEeqnarray*}
holds for any stopping time $\tau$. Taking the infimum over all stopping times $\tau$ yields 
$v(0,x_1)\geq v(0,x_2)$. In particular, if $v(0,x_1)<0$, then also $v(0,x_2)<0$, which shows that $\C$ has the claimed form.
The monotonicity of $b$ is immediate from Corollary~\ref{time-increase}.
%\item[(ii)]

(ii).
With $t\geq T$, we have $\Psi^2(t+s,\cdot)\leq c$  for all $s\geq 0$ by Proposition~\ref{T:psigr}(\ref{T:psigr1}), and the claim follows from Remark \ref{stop_now}. For $t<T$, on the other hand, there are points $x\in \mathcal{I}_\mu$ with $\Psi^2(t,x)>c$, so the respective claim follows from Lemma~\ref{T:locallygood}.

(iii). The upper semi-continuity of
$v$ and the monotonicity of $b$ imply $b(t)=b(t+)$ for all $t\geq 0$.

Next assume that $x\mapsto \Psi(t,x)$ is strictly increasing, and that $b(t_1-)<b(t_1)$ for some $t_1>0$. Since $(t_1,b(t_1))\in\D$, it follows from Lemma~\ref{T:locallygood} that $\Psi^2(t_1,b(t_1))\leq c$. Consequently, there exists an 
$\ep>0$ and a rectangle 
$\mathcal R=(t_0,t_1)\times(x_1,x_2)$ with $t_0<t_1$ and $b(t_1-)\leq x_1<x_2\leq b(t_1)$ such that
$\mathcal R\subseteq \C$ and $\Psi^2\leq c-\ep$ on $\mathcal R$. Moreover, 
$v(t_0,x_1)\leq v<0$ on $\mathcal R$. For a starting point 
$(t,x)\in\mathcal R$, define
\[\tau^{t,x}_{\mathcal R} \,:= \,\inf \big\{s\geq 0:(t+s,\widehat X^{t,x}(t+s))\notin\mathcal R \big\}\]
to be the first exit time from $\mathcal R$.
Since $\mathcal R\subseteq \C$, the process 
\[
v\left(t+\big( s\wedge\tau^{t,x}_{\mathcal R} \big),\widehat X^{t,x} \big(t+ \big( s\wedge\tau^{t,x}_{\mathcal R} \big) \big)\right)
+\int_0^{\,s\wedge\tau^{t,x}_{\mathcal R}} \left(c-\Psi^2 \big(t+ \theta,\widehat X^{t,x}(t+\theta) \big)\right)\ud \theta  \,, \quad s \ge 0 
\]
is a martingale by optimal stopping theory, and
\begin{eqnarray*}
	v(t,x) &=& \E\left[ v\left(t+\tau^{t,x}_{\mathcal R},\widehat X^{t,x}(t+\tau^{t,x}_{\mathcal R})\right)+
	\int_0^{\tau_{\mathcal R}^{t,x}}\left(c-\Psi^2 \big(t+s,\widehat X^{t,x}(t+s) \big)\right)\ud s\right]\\
	&\geq& \E\left[\mathbf{ 1}_{\{\tau^{t,x}_{\mathcal R}\geq t_1-t\}}\int_0^{t_1-t} \left(c-\Psi^2 \big(t+s,\widehat X^{t,x}(t+s) \big) \right)\ud s \right]+ v(t_0,x_1)\, \mathbb P \big(\tau^{t,x}_{\mathcal R}<t_1-t \big)\\
	&\geq& \ep (t_1-t)\,\mathbb P \big(\tau^{t,x}_{\mathcal R}\geq t_1-t \big)+ v(t_0,x_1) \, \mathbb P \big(\tau^{t,x}_{\mathcal R}<t_1-t \big).
\end{eqnarray*}
Here the first term is of size $\ep(t_1-t)$ for $t$ close to $t_1$, whereas the probability $\mathbb P(\tau^{t,x}_{\mathcal R}<t_1-t)$ is of order $o(t_1-t)$ as $t\to t_1$. 
Consequently, for each $x\in(x_1,x_2)$ there exists $t$ close to $t_1$ such that
$v(t,x)> 0$, which is a contradiction. This proves that $b(t_1-)=b(t_1)$, so $b$ is continuous.
%\end{itemize}
\end{proof}

\begin{remark}
{\rm
There is an analogue of Proposition~\ref{one-sided} for problems in which the function $x\mapsto \Psi(t,x)$ is non-increasing for every fixed $t \ge 0$. 
Arguing exactly as above, this condition implies the existence of a non-increasing boundary $b$ such that 
$$\mathcal C \,= \,\big\{(t,x):x<b(t) \big\} \,.$$
}
\end{remark}

\subsection{A case with a one-sided stopping region: the absolute value of a normal distribution}
Let us consider a case where the prior belief is represented by the absolute value of a normally distributed random variable 
with mean $0$ and variance $\sigma^2$, i.e.,
\[
\mu(\vd u)=\sqrt{\frac{2}{\pi\sigma^2}} \,\, \exp\left\{-\frac{u^2}{2\sigma^2}\right\}\vd u, \quad\quad u\geq 0.
\]
Then $\mathcal{I}=(0,\infty),$  and determined computation gives
\[
H(t,y)=\frac{\sigma^2}{1+\sigma^2t}\left(1- z\, \frac{\varphi(z)}{ \Phi (z)}-\frac{\varphi^2(z)}{\Phi^2(z)}\right) \bigg|_{z= Z(t,y)}, \quad \text{for} \quad Z(t,y) := \frac{\sigma y }{\,\sqrt{1+\sigma^2t}\,}
\]
and
\[
\varphi(b)=\frac{1}{\sqrt{2\pi}}\exp\{-b^2/2\}, \qquad \Phi(a)=\int_{-\infty}^a\varphi(b)\,\vd b
\]
for the function of \eqref{E:H(t,y)}. Note that this function satisfies $$\lim_{y\to\infty}H(t,y)\,=\,\frac{\sigma^2}{\,1+\sigma^2\,t\,}$$ for $t\geq 0$, very much in accordance with Section 3.2.   
Furthermore, $\Psi(t,\cdot)$ is increasing if and only if $H(t,\cdot)$ is increasing, and 
\begin{eqnarray*}
D H(t,y) &=&  \frac{\sigma^3\varphi(z)}{ \, (1+\sigma^2t)^{3/2}\, \Phi(z)\,}\, \Big( z^2-1 + 3z \frac{\varphi(z)}{\Phi(z)}+2\frac{\varphi^2(z)}{\Phi^2(z)}\Big)\bigg|_{z= Z(t,y)}.
\end{eqnarray*}
To see that $DH\geq 0$, we follow an argument from \cite{S}. It suffices to check that 
\[
f(z) := z^2 + 3z \frac{\varphi(z)}{\Phi(z)}+2\frac{\varphi^2(z)}{\Phi^2(z)} = \left(z+2\frac{\varphi(z)}{\Phi(z)}\right)\left(z+ \frac{\varphi(z)}{\Phi(z)}\right) \geq 1.
\]
Straightforward calculations give
\begin{eqnarray}\label{fprime}
 f'(z) &=& 2\left(\frac{\varphi(z)}{\Phi(z)} +z\right)\left(1- z \frac{\varphi(z)}{ \Phi (z)}-\frac{\varphi^2(z)}{\Phi^2(z)}\right)+\frac{\varphi(z)}{\Phi(z)}(1-f(z))\\
 \notag 
 &>& \frac{\varphi(z)}{\Phi(z)}(1-f(z))
\end{eqnarray}
at all points $z$. 
However, it is clear that $\lim_{z\to\infty} f(z)=\infty$, and using the expansion 
\[
\Phi(z) =\frac{\varphi(z)}{-z} \bigg(1-\frac{1}{z^2}+o \big(1/z^2 \big) \bigg)
\]
for $z<0$ yields $\lim_{z\to -\infty} f(z)=1$. Therefore, 
if there is a finite root of the equation $f(z)=1$, then there exists a finite $z_0$ with $f(z_0)\leq 1$ and $f'(z_0)=0$, which contradicts \eqref{fprime}.
Therefore, $f\geq 1$ so $DH\geq 0$.

It now follows from Proposition~\ref{one-sided} that the continuation region is one-sided and given by
\[
\mathcal C \,=\, \big\{(t,x)\in[0,T)\times (0,\infty):x>b(t) \big\}
\]
for some continuous, non-decreasing function $b:[0,T)\to[0,\infty)$ with $b(T):=\lim_{t\uparrow T}b(t)=\infty$, where
$T=\left(\frac{1}{\sqrt{ c}}-\frac{1}{\sigma^2}\right)^+$.

%%%%%%%%
\section{Symmetric prior distributions}
%%%%%%%%
\label{sec6}

In this section we consider the special case when $\mu$ is symmetric around the origin with $\mathcal{I}_\mu=(-a,a)$ in \eqref{E:supp},  for some $a\in(0,\infty]$. 
Then the functions  $\Psi$ and $v$ are also symmetric around
the origin,  in the sense that $\Psi(t,x)=\Psi(t,-x)$ and $v(t,x)=v(t,-x)$. Consequently, the optimal stopping problem can be re-written in terms
of the reflected diffusion $Z=Z^{t,x}=\vert \widehat X^{t,x}\vert$ as
\begin{IEEEeqnarray}{rCl} 
\label{E:vZ}
v(t,x)=\inf_{\tau\in\mathcal T%^{\hat W}
}
\E\left[\int_0^\tau  \Big( c-\Psi^2 \big(t+s,Z^{t,x}(t+s)\big) \Big)\ud s\right], \quad (t,x)\in[0,\infty)\times [0,a).~~~~~~
\end{IEEEeqnarray}

\begin{proposition}\label{symmetry}
Assume that $\mu$ is symmetric around the origin.
\begin{itemize}
\item[(i)]
Assume that, for every fixed time $t\geq 0$,  the function $\Psi(t,\cdot):[0,a)\to[0,\infty)$ is non-decreasing. 
Then there exist
a point $t_0\geq 0$ and a non-decreasing boundary $b:[t_0,\infty)\to [0,a]$ such that 
\[\mathcal C=([0,t_0)\times \mathcal I_\mu)\cup \{(t,x)\in[t_0,\infty)\times \mathcal I_\mu:\vert x\vert >b(t)\}.\]
\item[(ii)]
Assume that,  for every fixed  time $t\geq 0$, the function $\Psi(t,\cdot):[0,a)\to[0,\infty)$ is non-increasing. Then there exists
a non-increasing boundary $\,b:[0,\infty)\to[0,a]\,$ such that $$\mathcal C\,=\, \big\{(t,x)\in[0,\infty)\times \mathcal I_\mu:\vert x\vert<b(t) \big\} .$$
\end{itemize}
\end{proposition}

\begin{proof}
Without loss of generality, we consider the initial time $t=0$.
For $x\geq 0$, let $(\widetilde Z,L)$ be the unique continuous process such that 
$L(0)=0$, $L$ is non-decreasing, $Z(0)=x$, $Z(s)\geq 0$ and
\[ \left\{\begin{array}{ll}
d\widetilde Z(s)=\Psi(s,\widetilde Z(s))\,d\widehat W(s) + dL(s)\\
\int_0^t \, \mathbf{ 1}_{\{\widetilde Z(s)=0\}}dL(s)=L(t).\end{array}\right.\]
Then $\widetilde Z$ is the reflected version of $\widehat X$, and the processes 
$\{Z(s),s\geq 0\}$ and $\{\widetilde Z(s), s\geq 0\}$ coincide in law.
Moreover,  by comparison we have that
$x_1\leq x_2$ implies that $\widetilde Z^{x_1}(t)\leq \widetilde Z^{x_2}(t)$ for all $t$. The proof then follows the proof of Proposition~\ref{one-sided}.
\end{proof}

%\subsection{An example where (i) of Proposition~\ref{symmetry} applies}

%Perhaps \footnote{~ I tried this, unsuccessfully. (IK)} 
%\[\mu(\vd u)=\left\{\begin{array}{ll}
%\frac{\varphi(u)}{2 \Phi(-\beta)}\, \ud u & \vert u\vert\geq \beta\\
%0 & |u |<\beta\end{array}\right.\]
%would do the job. \marginal{EE: I will check this.}

\subsection{Symmetric Gaussian mixtures}

We end the article with a study of the case when the prior is given by 
a symmetric Gaussian mixture.
More precisely, let $\mu$ be given by
\[\mu(\vd u)=\frac{1}{2\sigma\sqrt{2\pi}}\left(\exp\left\{\frac{-(u-m)^2}{2\sigma^2}\right\} +\exp\left\{\frac{-(u+m)^2}{2\sigma^2}\right\}\right) \vd u\]
with $m\in(0,\infty)$ and $\sigma>0$, i.e., a mixture of two Gaussians $N(m,\sigma)$ and $N(-m,\sigma)$. Then
\begin{eqnarray*}
F(t,y) &=&  \frac{1}{%2\sqrt{2\pi\sigma^2}
\sqrt{1+ \sigma^2 t}}\bigg( \exp\bigg\{ \,- \frac{1}{2\sigma^2} \, \bigg( \frac{(m + \sigma^2 y)^2}{1+\sigma^2 t}- m^2\bigg) \bigg\} \\
&& \hspace{30mm}+
  \exp\bigg\{ \,- \frac{1}{2\sigma^2} \, \bigg( \frac{(-m + \sigma^2 y)^2}{1+\sigma^2 t}- m^2\bigg) \bigg\}
\bigg),
\end{eqnarray*}
and straightforward calculations yield
\[
H(t,y)=\frac{\sigma^2}{1+\sigma^2t}+\frac{4m^2}{(1+\sigma^2t)^2}\left(\exp\left\{\frac{my}{1+\sigma^2t}\right\} + 
\exp\left\{\frac{-my}{1+\sigma^2t}\right\}\right)^{-2}.
\]
It follows that $\Psi(t,\cdot)$ is decreasing on $[0,\infty)$ and satisfies 
\[
\Psi(t,0)=\frac{\sigma^2}{1+\sigma^2t}+\frac{m^2}{(1+\sigma^2t)^2}\,, \qquad \Psi(t,\infty)=\frac{\sigma^2}{1+\sigma^2t}.
\]
Consequently, by (ii) of Proposition~\ref{symmetry}, there exists a non-increasing boundary $b:[0,\infty)\to[0,\infty]$
such that 
$$\mathcal C\,=\, \big\{(t,x)\in[0,\infty)\times \mathcal I_\mu:\vert x\vert<b(t) \big\} .$$
Furthermore, $b(t)=\infty$ for $t\in \big[0,(c^{-1/2}-\sigma^{-2})^+\big)$ and $b(t)=0$ for 
\[t\geq  \frac{1}{2\sqrt c}\left(1-2\sigma^{-2}\sqrt c+\sqrt{1+4m^2\sigma^{-4}c^{1/2}}\,\right)^+.\]

\bigskip

%%%%%%%%%
\section*{Dedication}
%%%%%%%%%

The second author is grateful to his teacher, Dr.\,V\'aclav E.\,\,Bene\v s,  for suggesting this problem to him more than twelve years ago   and for urging him to make progress on it ever since. 

We dedicate this paper to Dr.\,\,Bene\v s on the occasion of his upcoming 90th birthday, with affection and respect.

%
%
%\begin{example}{\bf (The two-point prior)} Suppose $\mu=\pi \delta_h + (1-\pi) \delta_l$, where $\pi \in (0,1)$, the symbols $\delta_l, \delta_h$ denote {the Dirac measures} at $l<0$ and $h>0$, respectively. Then $ \psi(t,x) =\frac{1}{\sigma} (h-x)(x-l)$. 
%\end{example}\label{Ex:B}
%
%\begin{example}{\bf (The normal prior)} Suppose $\mu$ is the normal distribution with mean $m$ and variance $\gm^2$. Then the conditional distribution $\P(\cdot\vert Y_t=y)=\mu_{t,y}$ is also normal but with mean $\frac{\sigma^2 m +\gamma^2y}{\sigma^2+t\gamma^2}$
%and variance $\frac{\sigma^2\gamma^2}{\sigma^2+t\gamma^2}$. Consequently, $\psi(t,x) = \frac{\sigma\gamma^2}{\sigma^2+t\gamma^2}$. \label{Ex:N}
%\end{example}


\newpage

\begin{thebibliography}{99}

\bibitem{BC09} 
\textsc{Bain, A. \& Crisan, D.} (2009) 
{\em Fundamentals of Stochastic Filtering.} Stochastic Modeling and Applied Probability, Volume 60. Springer, New York.

\bibitem{CNS}
\textsc{Cetin, U., Novikov, A. \& Shiryaev, A.N.} (2013)
Bayesian sequential estimation of a drift of fractional Brownian motion. 
{\em Sequential Anal.} {\bf 32}, no. 3, 288-296. 

%\bibitem{BS02} Borodin, A. N. and Salminen P. Handbook of Brownian Motion — Facts and Formulae. Second edition.  \emph{Birkh\"auser}, 2002.

%\bibitem{BC}
%Breakwell, J. and Chernoff, H. Sequential tests for the mean of a normal distribution. II. (Large t). {\em Ann. Math. Statist}. 
%35 (1964), 162-173.
%
%\bibitem{C}
%Chernoff, H. Sequential tests for the mean of a normal distribution. {\em Proc. 4th Berkeley Sympos. Math. Statist. and Prob.} (1961), 
%Vol. I, 79-91, Univ. California Press, Berkeley, Calif.
%
%\bibitem{C65a}
%Chernoff, H. Sequential tests for the mean of a normal distribution III (small t). {\em Ann.Math. Statist.} 36 (1965), 28-54.
%\bibitem{C65b} Chernoff, H. Sequential tests for the mean of a normal distribution IV (discrete case).
%{\em Ann. Math. Statist.} 36 (1965), 55-68.
%
%\bibitem{E}
%Ekstr\"om, E. Properties of American option prices. {\em Stochastic Process. Appl.} 114 (2004), no. 2, 265-278.

\bibitem{EV}
\textsc{Ekstr\"om, E. \& Vaicenavicius, J.}  (2016)
Optimal liquidation of an asset under drift uncertainty. {\it SIAM J. Financial Math.} {\bf 7}, no. 1, 357-381.


\bibitem{ElK}
\textsc{El\,Karoui, N.}  (1981)
Les Aspects Probabilistes du Contr\^ole Stochastique. {\it Lecture Notes in Mathematics} {\bf 876},  73-238.

\bibitem{Fak}
\textsc{Fakeev, A.G.}  (1971)
Optimal stopping of a Markov process. {\it Theory Probab. Appl.} {\bf 15},  324-331.

%\bibitem{J}
% Jacka, S. Optimal stopping and the American put. {\em Math. Finance} 1 (1991), 1-14.

%\bibitem{JT}
%Janson, S. and Tysk, J. Volatility time and properties of option prices. 
%{\em Ann. Appl. Probab}. 13 (2003), no. 3, 890-913. 

%\bibitem{KS} 
%Karatzas, I. and Shreve, S. {\em Brownian motion and stochastic calculus. Second edition.} Graduate Texts in Mathematics, 113. 
%Springer-Verlag, New York, 1991.

%\bibitem{KS2}
%Karatzas, I. and Shreve, S. {\em Methods of mathematical finance}. 
%Applications of Mathematics (New York), 39. Springer-Verlag, New York, 1998.

%\bibitem{aK08} 
%A. Klenke, \emph{Probability Theory: A Comprehensive Course}  
%Universitext. Springer-Verlag London, Ltd., London, 2008.

%\bibitem{tL88} Lai, T. L. Nearly optimal sequential tests of composite hypotheses.{ \em Ann. Statist.} 16 (1988),
%856-886.

%\bibitem{tL97} Lai, T. L. On optimal stopping problems in sequential hypothesis testing, {\em Statistica Sinica} 7 (1997), 33-51.

%\bibitem{bO} 
%\O ksendal, B. {\em Stochastic differential equations. 
%An introduction with applications. Sixth edition.} Universitext. {\em Springer-Verlag, Berlin}, 2003.
%
%\bibitem{P2}
%Peskir, G. On the American option problem. {\em Math. Finance} 15 (2005), no. 1, 169-181.
%\bibitem{B04} Back, K. Incomplete and asymmetric information in asset pricing theory.
%Stochastic Methods in Finance,
%Lecture Notes in Mathematics, vol. 1856, 2004, 1-25. 
%
%\bibitem{BL97} Beibel, M., Lerche, H.R. A New Look at Optimal Stopping Problems related to Mathematical Finance. Statistica Sinica, vol.~7, no.~1, 1997, 93-108.
%
%\bibitem{BDL10}  Bj\"ork, T., Davis, M. H. A., Land\'en, C. Optimal investment under partial information.
%Mathematical Methods of Operations Research, vol.~71, no.~2, 2010, 371-399. 
%
%\bibitem{BH11} Bush, N., Hambly, B.M., Haworth, H., Jin, L., Reisinger, C. Stochastic evolution equations in portfolio credit modelling. SIAM Journal on Financial Mathematics, vol.~2, no.~1, 2011, 627-664. 
%
%\bibitem{DMV}
%D\'ecamps, J.-P., Mariotti, T., Villeneuve, S. Investment timing under incomplete information. Math. Oper. Res., vol.~30, no. 2, 2005, 472-500.
%
%\bibitem{DF}
%Dothan, M., Feldman, D. Equilibrium interest rates and multiperiod bonds in a partially observable economy. Journal of Finance, vol.~41, no.~2, 1986, 369-382.
%
%\bibitem{TP} du Toit, J., Peskir, G. Selling a stock at the ultimate maximum. Annals of Applied Probability, vol.~19, no.~3, 2009, 983-1014.
%
%\bibitem{EL11} Ekstr\"om, E., Lu, B. Optimal selling of an asset under incomplete information.
%International Journal of Stochastic Analysis, vol. 2011, 2011, ID 543590. 
%
%\bibitem{ET08} Ekstr\"om, E., Tysk, J. Convexity theory for the term structure equation. Finance and Stochastics. vol.~12, no.~1, 2008, 117-147.
%
%\bibitem{EHH} Evans, J.,~Henderson, V.,~Hobson, D. Optimal timing for an indivisible asset sale. Mathematical Finance, vol.~18, no.~4, 2008, 545-567.
%
%\bibitem{G}
%Gapeev, P. Pricing of perpetual American options in a model with partial information. 
%International Journal of Theoretical and Applied Finance, vol.~15, no.~1, 2012, ID 1250010.
%
%\bibitem{J}
%Jacka, S. Optimal stopping and the American put. Mathematical Finance 1 (1991), no.~2, 1-14.

\bibitem{JT}
\textsc{Janson, S. \& Tysk, J.} (2003) Volatility time and properties of option prices. {\it Ann. Appl. Probab.} {\bf  13}, no. 3, 890-913.

%\bibitem{K}
%  Karatzas, I. Bayesian sequential least-squares estimation under partial observations.

 
\bibitem{KS1} 
 \textsc{Karatzas, I. \& Shreve, S.E.} (1991) {\em Brownian Motion and Stochastic Calculus.}  Second Edition,  Graduate Texts in Mathematics, Vol. 113. Springer-Verlag, New York.
%
%\bibitem{KS2}
%Karatzas, I., Shreve, S. {\em Methods of Mathematical Finance}. 
%Applications of Mathematics, 39. Springer-Verlag, New York, 1998.
%
%\bibitem{L95}
%Lakner, P. Utility maximization with partial information. Stochastic Processes and their Applications, vol.~56, no.~2, 1995, 247-249.
%
%\bibitem{L98}
%Lakner, P. Optimal trading strategy for an investor: the case of partial information. Stochastic Processes and their Applications, vol.~76, no.~1, 1998, 77-97.
%
%\bibitem{M} Monoyios, M. Optimal investment and hedging under partial and inside information.
%{ Advanced Financial Modelling, Radon Series on Computational and Applied Mathematics}, vol.~8, 2009, 371-410.
%
%\bibitem{MN11} Monoyios, M., Ng, A.~Optimal exercise of an executive stock option by an insider. International Journal of Theoretical and Applied Finance, vol.~14, no.~1, 2011, 83-106.
%

\bibitem{bO07}
\textsc{\O ksendal, B.} (2007) { \em Stochastic Differential Equations: An Introduction with Applications.}  Sixth Edition, Springer-Verlag, New York. 

\bibitem{RY}
\textsc{Revuz, D. \& Yor, M.} (1999) {\em Continuous Martingales and Brownian Motion.}  Third Edition, Grundlehren der Mathematischen Wissenschaften, Vol. 293. Springer-Verlag, Berlin.

\bibitem{RS73} 
 \textsc{Robbins, H. \& Siegmund, D.} (1973) Statistical tests of power one, and the integral representation of solutions of certain parabolic differential equations. {\it Bull. Inst. Math.  Acad. Sinica} (Taipei) {\bf  1}, 93-120.


\bibitem{S}
\textsc{Sampford, M.R.} (1953)
Some inequalities on Mill's ratio and related functions.
{\em Ann. Math. Statistics} {\bf 24}, 130-132. 
 
 \bibitem{Shir}
\textsc{Shiryaev, A.N.} (1978) { \em Optimal Stopping Rules.}   Springer-Verlag, New York. 

%


 \bibitem{vM76} 
 \textsc{van\,Moerbeke, P.} (1976) Optimal stopping and free boundary problems. {\it Arch. Rat'l. Mech. Anal.} {\bf  60}, 101-148.


 \bibitem{W44} 
 \textsc{Widder, D.V.} (1944) Positive temperatures on an infinite rod. {\it Trans. Amer. Math. Soc.} {\bf  75}, 510-525.

\end{thebibliography}
\end{document}